\documentclass{amsart}

\usepackage[english]{babel}
\usepackage{todonotes}
\usepackage{amsmath,amsthm,amssymb,amsfonts}
\usepackage{graphicx}

\newcommand{\R}{\mathbb{R}}

\newcommand{\N}{\mathbb{N}}
\renewcommand{\S}{\mathbb{S}}
\newcommand{\Z}{\mathbb{Z}}
\renewcommand{\S}{\mathbb{S}}

\newcommand{\I}{\mathcal I}

\newcommand{\eps}{\varepsilon}

\newcommand{\uinf}{v}

\DeclareMathOperator{\tr}{Tr}

\newtheorem{proposition}{Proposition}
\newtheorem{lemma}{Lemma}

\newtheorem{theorem}{Theorem}
\newtheorem{claim}{Claim}

\theoremstyle{definition}

\theoremstyle{remark}
\newtheorem{remark}{Remark}

\newtheorem{example}{Example}

\usepackage{color}
\definecolor{linkcol}{rgb}{0,0,0.4} 
\definecolor{citecol}{rgb}{0.5,0,0} 
\definecolor{darkred}{rgb}{0.75,0,0.1} 
\definecolor{darkblue}{rgb}{0.1,0,0.55}

\author{Guy Barles}
\address{Laboratoire de Math\'ematiques et Physique Th\'eorique,
    CNRS UMR 7350, F\'ed\'eration Denis Poisson,
    Universit\'e Fran\c{c}ois Rabelais, 
    Parc de Grandmont,
    37200 Tours, France} 
\email{barles@lmpt.univ-tours.fr}

\author{Emmanuel Chasseigne}
\address{Laboratoire de Math\'ematiques et Physique Th\'eorique,
    CNRS UMR 7350, F\'ed\'eration Denis Poisson,
    Universit\'e Fran\c{c}ois Rabelais, 
    Parc de Grandmont,
    37200 Tours, France}
\email{emmanuel.chasseigne@lmpt.univ-tours.fr}

\author{Adina Ciomaga} 
\address{Department of Mathematics, 
    University of Chicago,
    5734 University Avenue, 
    Chica\-go, IL 60637, USA}
\email{adina@math.uchicago.edu}

\author{Cyril Imbert} \address{CNRS, Laboratoire d'Analyse et de
  Math\'ematiques Appliqu\'ees, UMR 8050, Universit\'e Paris-Est
  Cr\'eteil, 61 avenue du g\'en\'eral de Gaulle 94010 Cr\'eteil cedex
  France} \email{cyril.imbert@u-pec.fr}

\title[Large Time Behavior of Integro-Differential Equations]
    {Large Time Behavior of Periodic Viscosity Solutions for
Uniformly Elliptic Integro-Differential Equations}
\date\today


\begin{document}

\begin{abstract}
In this paper, we study the large time behavior of solutions of a
class of parabolic fully nonlinear integro-differential equations in a
periodic setting. In order to do so, we first solve the \emph{ergodic
  problem} (or \emph{cell problem}), i.e. we construct solutions of
the form $\lambda t + v(x)$. We then prove that solutions of the
Cauchy problem look like those specific solutions as time goes to
infinity.  We face two key difficulties to carry out this classical
program: (i) the fact that we handle the case of ``mixed operators''
for which the required ellipticity comes from a combination of the
properties of the local and nonlocal terms and (ii) the treatment of
the superlinear case (in the gradient variable). Lipschitz estimates
previously proved by the authors (2012) and Strong Maximum principles
proved by the third author (2012) play a crucial role in the analysis.
\end{abstract}

\maketitle

\section{Introduction}\label{sec:Intro}

In this paper, we provide new results on the large time behavior of
viscosity solutions for parabolic integro-differential equations (PIDE
in short). 

\subsection{A model equation}

In order to describe our approach and our results, we consider the
following model example
\begin{equation}\label{model}
\partial_t u - \Delta_{x_1} u + (-\Delta)_{x_2}^\beta u +|Du|^m +
b_1(x_1) |D_{x_1}u| + b_2(x_2) |D_{x_2}u| = f(x) \;
\end{equation}
where $u:\R^d \times [0,+\infty) \to \R$ is the unknown function
  depending on $x=(x_1,x_2)\in \R^d$ with $x_1 \in \R^{d_1},$ $x_2\in
  \R^{d_2}$, $d_1+d_2 = d$, and $t\geq 0$. $\partial_t u$ denotes the
  derivative of $u$ with respect to $t$, while
  $Du=(D_{x_1}u,D_{x_2}u)$ stands for its gradient with respect to the
  space variable $x$. The operator $\Delta_{x_1} u$ is the usual
  Laplacian with respect to the $x_1$-variable, while
  $(-\Delta)_{x_2}^\beta u$ denotes the fractional Laplacian of
  exponent $\beta\in(1,2)$ with respect to the $x_2$-variable

\[(-\Delta)_{x_2}^\beta u (x,t) = 
\int_{z_2 \in \R^{d_2}} \big( u(x_1, x_2 + z_2,t) - u(x,t) -
D_{x_2}u(x,t)\cdot z_2
1_{B^{d_2}}(z_2)\big)\frac{dz_2}{|z_2|^{d_2+\beta}}\; ,\] where
$D_{x_2}u$ is the gradient of $u$ with respect to the $x_2$-variable
and $B^{d_2}$ is the unit ball in $\R^{d_2}$.  Finally, we assume that
$m\geq 1$ and $f$, $b_i$ ($i=1,2$) are real-valued, Lipschitz
continuous functions which are respectively $\Z^d$- and $\Z^{d_i}$-
periodic.  Because of this last assumption, the solution is expected
to be $\Z^d$-periodic if the initial datum is.

\subsection{Aim}

For such a PIDE, our aim is to show that, for large times ($t \to
+\infty$), the solution $u$ asymptotically behaves like $\lambda t+
v(x)$ where $(\lambda,v)\in\R\times C^0(\R^d;\R)$ is a solution of the
associated ergodic (or additive eigenvalue) problem which, for
(\ref{model}), reads
\begin{equation}\label{erg-model} 
- \Delta_{x_1} v + (-\Delta)_{x_2}^\beta v +|Dv|^m + b_1(x_1)
|D_{x_1}v| + b_2(x_2) |D_{x_2}v|= f(x)- \lambda .
\end{equation}
A key result in this direction is that there exists a unique $\lambda
\in \R$ such that this ergodic problem has a periodic (Lipschitz)
continuous solution $v$. With such a result in hand, one has to prove the
convergence, namely that for every solution $u$ of \eqref{model},
there exists a solution $(\lambda,v)$ of the ergodic problem such that
$u(x,t)-(\lambda t+ v(x)) \to 0$ as $t\to +\infty$, uniformly in $x\in
\R^d$.

To do so, we follow a by-now rather classical method which was
systematically developed in \cite{bs01}. To carry out this method, the
two key ingredients are estimates on the modulus of continuity
(Lipschitz estimates in our case),  and a Strong Maximum
  Principle, both for equations (\ref{model}) and (\ref{erg-model}).

These needed results were obtained in previous papers. Lipschitz and
H\"older estimates were obtained in \cite{bcci12} where an emphasis
was made on ``mixed operators'', i.e. on equations like (\ref{model})
where the ``uniform ellipticity'' comes from both integral and
differential terms, namely the $\Delta_{x_1}$ and
$(-\Delta)_{x_2}^\beta$ terms. This particular form of the equation
creates also difficulties for the Strong Maximum (or Comparison)
Principle, see \cite{c12}.

In addition to this specific difficulty coming from ``mixed
operators'', we also want to handle the ``superlinear case'', namely
the case when $m>1$ in (\ref{model}). This requires additional work
and ideas both for solving (\ref{erg-model}) and for the convergence
proof since one needs Lipschitz estimates to ``linearize'' this term.
For this reason, we distinguish two cases: a sublinear and a
superlinear one, with respect to the gradient growth.

In the sublinear case the estimates on the modulus of continuity come
from the ellipticity of the equation. Indeed, as shown in
\cite{bcci12}, even though the equation is completely degenerate in
the local term and in the nonlocal term, their combination render the
diffusion uniformly elliptic. In the superlinear case, although most
of these estimates are derived through the same type of arguments
under suitable structure conditions on the nonlinearities, there are
some situations where they come from the gradient term (cf. the proof
of Lemma~\ref{lemma:unif_cont}).

It is worth pointing out that some of our results (in particular in
the sublinear case) could be proved in an easier way since they do not
require such Lipschitz estimates but we have chosen to systematically
use them in order to unify the paper and to keep it with a reasonable
length.

\subsection{The general framework}

Let us now present the general framework of our analysis.  Although
not the most general one, we have chosen this framework since it
carries the key difficulties. Extensions to a larger framework are
given in Section~\ref{sec:extensions}.  We consider parabolic PIDE of
the form
\begin{align}\label{eq:main} \partial_t u +&
  F_1(x_1,D_{x_1}u, D_{x_1x_1}^2u, \mathcal I_{x_1}[x,u]) &\\ +&
  F_2(x_2,D_{x_2}u, D_{x_2x_2}^2u, \mathcal I_{x_2}[x,u]) + H(Du) =
  f(x), &\hbox{ in  }\R^d \times (0,+\infty) \nonumber 
\end{align}
(where $D_{x_i x_i}^2 u$ denotes the Hessian matrix with respect to
$x_i$, $i=1,2$) subject to the initial condition
\begin{equation} \label{eq:ic} u(x,0)
    =  u_0(x),\quad   \hbox{in   }\R^d 
\end{equation} 
where $u_0$ and $f$ are $\Z^d$-periodic functions, $H$ is a continuous
function in $\R^d$ and $F_1, F_2$ are nonlinear ``elliptic'' terms
(precise assumptions are given in
Section~\ref{sec:assumptions}). We point out that each nonlinear term
involves second-order derivatives and a nonlocal operator of
``order'' $\beta \in (1,2)$. These operators $ \mathcal I_{x_i}$ are
of \emph{L\'evy-It\^o type}: if $\varphi : \R^{d_i} \to \R$ is a
smooth bounded function, then

    \[ \mathcal{I}_{i}[\varphi](x_i) 
    = \int_{z_i\in \R^{d_i}} \big( \varphi (x_i +
    j_i(x_i,z_i)) - \varphi(x_i) - D\varphi(x_i)\cdot
    j_i(x_i,z_i)1_{B^{d_i}}(z_i)\big)d\mu_i(z_i)\] 
where $j_i$ are the \emph{jump functions} and $\mu_i$ the
    \emph{L\'evy measures} (here we also refer the reader to
    Section~\ref{sec:assumptions} for precise assumptions). Then we set 
\[ \mathcal I_{x_1} [x,u]:= \mathcal
    {I}_1[u(\cdot,x_2,t)](x_1)\quad \hbox{and}\quad \mathcal I_{x_2} [x,u]:=
    \mathcal {I}_2[u(x_1,\cdot,t)](x_2)\; .\]

We recall that these operators are ``natural'' generalizations of
diffusion operators of the form $\frac12 \tr( \sigma(x)\sigma^T(x)
D^2u)$ where $\sigma$ is the diffusion matrix and $\sigma^T$ denotes
its transpose; in particular, they can be characterized as the
infinitesimal generators of some solutions of stochastic differential
equations driven by a general L\'evy process, instead of a Brownian
motion; see \cite{applebaum} for more details.

In the case of (\ref{model}), we have
\begin{align*} 
F_1(x_1,D_{x_1}u, D_{x_1x_1}^2u, \mathcal I_{x_1}[x,u])&:=- \Delta_{x_1} u + b_1(x_1) |D_{x_1}u|\; ,\\
F_2(x_2,D_{x_2}u, D_{x_2x_2}^2u, \mathcal I_{x_2}[x,u]) &:=(-\Delta)_{x_2}^\beta u+ b_2(x_2) |D_{x_2}u| \; ,\\
 H(Du) &:=  |Du|^m  \; .
\end{align*}
By choosing such a framework, we want to shed some light on the fact
that each nonlinear term can have different forms of (local or
nonlocal) diffusions in the different sets of variables while the
$H$-term carries the possible super-linearity (see Assumptions (H-a)
and (H-b) below). Of course this simplified framework can be
generalized and we refer the reader to Section~\ref{sec:extensions}
devoted to the extensions to see how this can be done for both the
$H$-term but also for $F_1$ and $F_2$.

\subsection{Known results}

The study of large time behavior of solutions of Hamilton-Jacobi and
fully nonlinear parabolic equations has attracted a lot of attention,
in different contexts. It follows closely the development of the
viscosity solution theory.  Indeed, the first result one could cite
about the large time behavior of solutions of first-order $x$
independent equations is already contained in \cite{Lions} (see also
\cite{b85}).

When studying the large time behavior of solutions of such equations,
one has first to construct solutions of the form $\lambda t + v(x)$;
equivalently, one has to solve a stationary equation depending on a
parameter $\lambda$ which is unknown. This is the so-called cell (or
ergodic) problem. Then one has to prove that the solution of the
Cauchy problem indeed converges towards $\lambda t + v(x)$ for some
solution $(\lambda,v)$ of the cell problem.
The first step was completed in the seminal (yet unpublished) work of
Lions, Papanicolaou and Varadhan \cite{lpv} about homogenization of
first-order Hamilton-Jacobi equations (in the case of coercive
Hamiltonians). Indeed, solving the cell problem is the way the
``effective'' (or averaged) Hamilton-Jacobi Equation is determined.

The second step was first completed, for Hamilton-Jacobi Equations, by
Namah and Roquejoffre \cite{NR97b} for equations with a particular
structures and then by Fathi \cite{F98} in the general case of convex
and coercive Hamilton-Jacobi Equations in the periodic setting. It is
worth pointing out that the proof of \cite{NR97b} is based on pde
methods while the results of \cite{F98} relies on dynamical systems
arguments (``Weak KAM method''). Afterwards J.-M. Roquejoffre \cite{R}
and A. Davini and A. Siconolfi in \cite{DS} refined the approach of
A. Fathi and they studied the asymptotic problem for Hamilton-Jacobi
Equations on a compact manifold.

Barles and Souganidis extended in \cite{BS00b} the previous results
and showed the asymptotic behavior under weaker convexity assumptions,
using viscosity solutions techniques. They also gave counterexamples
in \cite{BS00a} on the asymptotic behavior of solutions of
Hamilton-Jacobi Equations, when the initial datum is not periodic
anymore. Motivated by the latest works, Ishii established in
\cite{I08} a general convergence result for Hamilton-Jacobi equations
on the whole space. For asymptotic behavior of solutions of various
boundary value problems for Hamilton-Jacobi Equations we refer to the
works \cite{II08,I11, M08, BM12, GQM12,BIM12} based either on weak KAM
theory or on PDEs techniques. Lastly, we mention that Dirr and
Souganidis showed in \cite{ds05} that the asymptotic behavior remains
true if one perturbes with additive noise viscous or non-viscous
Hamilton-Jacobi equations, periodic in space.

As we already mentioned, our approach follows the ideas systematically
developed by Barles and Souganidis in \cite{bs01}, where they
described the long time behavior of space-time periodic solutions of
quasilinear PDEs.  This behavior has also been established for
semilinear equations, with methods of degree theory, by Namah and
Roquejoffre in \cite{NR97a}. Long time behavior and ergodic problems
for second order PDEs with Neumann boundary conditions have been
studied in the series of papers \cite{gl05,L08,BLLS08}. Recently,
viscous Hamilton-Jacobi equations have been treated by Tchamba in the
superquadratic case \cite{T10}, and Barles, Porretta and Tchamba in the
subquadratic case \cite{BPT10}. 

We already mentioned that, when studying homogenization of
Hamilton-Jacobi or fully nonlinear elliptic equation, it is necessary
to solve a cell problem.  As far as nonlinear integro-differential
equations with periodic data are concerned, we can first mention a
series of papers devoted to the homogenization of dislocation
dynamics, see e.g. \cite{imr,fim}. Some results were also obtained for
linear integro-differential equations with periodic data in
\cite{ccr,arisawa} by analytical methods and in
\cite{tomisaki,ft94,franke,franke2} by probabilistic ones.
Homogenization of Markov processes governed by certain Levy operators
was discussed by Horie Inuzuka and Tanaka in \cite{hit77} and general
results on their convergence in law were established in \cite{ft94}
(see also \cite{tomisaki}).  In \cite{franke}, the author studies the
convergence in law of a rescaled solution of a stochastic differential
equation driven by an $\beta$-stable L\'evy process, $\beta >1$.  The
author points out that such an homogenization question was raised in
\cite[p.~531]{blp}. It is also mentioned in the introduction of
\cite{rv09} that they are very few such results.

For nonlinear equations, Schwab \cite{s10}
established homogenization results for a large class of equations.  In
the previously mentioned papers, both in the linear and the nonlinear
cases, an ergodic problem has to be solved. In \cite{franke}, since
the equation is linear, linear equation techniques are used such as
the study of the resolvent. In papers such as
\cite{imr,fim,arisawa,s10}, viscosity solutions / maximum principle
techniques are used.

\subsection{Organization of the article}

The paper is organized as follows. In Section~\ref{sec:assumptions},
assumptions on the nonlocal operator $\I$ and the nonlinearities $H,
F_1,F_2$ are given.  We also point out which known results from
\cite{bi08,c12,bcci12} can be used with such a set of assumptions.  In
Section~\ref{sec:ergodic}, we solve the stationary ergodic problem. In
Section~\ref{sec:longtime}, we state and prove the convergence result
for solutions of the Cauchy problem. In Section~\ref{sec:extensions},
we explain how to extend the previous results to even more
equations. In Section~\ref{sec:ex}, we give examples of applications
of our results on specific equations.

\section{Assumptions}\label{sec:assumptions}

Before stating our assumptions, we want to point out that, in order to
avoid too many technicalities, we are going to use simplified
assumptions for the main results, in particular some strong
assumptions on the homogeneity of $F_1$ and $F_2$. Then in Section
\ref{sec:extensions} we explain how to extend our results to more
general situations, under weaker homogeneity and growth assumptions.

Thus, the set of hypotheses below may not always seem consistent when
considered as a whole. Typically, assuming the homogeneity assumption
(F0) would simplify the general growth hypothesis (F2) a lot. But we
keep these assumptions as such since in the extension section, we shall
use some of them in their general versions.

We begin with assumptions on the singular measures and jump functions.
\begin{itemize}
\item[(M1)] (\textit{Integrability}) For $i=1,2$, $\mu_i$ is a L\' evy
  measure, i.e. there exists $\tilde C_\mu >0$ such that,
 \[\int_{\R^{d_i}}\min (|z_i|^2,1)d\mu_i(z_i) \leq \tilde C_\mu \quad \hbox{(for $i=1,2$)}.\]
\item[(M2)] (\textit{Regularity of the measures}) There exists $\bar C_\mu$ and, for $i=1,2$, there exists $\beta_i \in(1,2)$ such that, for $\delta>0$ small enough
\[\int_{B^{d_i}\setminus B^{d_i}_\delta} |z_i| d\mu_i(z_i) \leq \bar C_\mu \delta^{1-\beta_i}\; ,\]
where $B^{d_i}_\delta$ is the ball centered at $0\in \R^{d_i}$ and of radius $\delta$.
\end{itemize}
\begin{itemize}
\item [(M3)] (\textit{Jump size})	
There exist two constants $c_0, C_0>0$ such that, for any $x\in\R^d$ and $i=1,2$
$$ c_0|z_i|\leq |j_i(x_i,z_i)|\leq C_0|z_i|, \quad \hbox{for all  } z_i \in B^{d_i}.$$
\item [(M4)] (\textit{Regularity of the jumps}) There exist a constant $\tilde C_0>0$ such that, for any $x_i,y_i\in\R^{d_i}$
$$ |j_i(x_i,z_i)-j_i(y_i,z_i)|\leq 	 \begin{cases}
	      \tilde C_0|z_i||x_i-y_i|, & \hbox{for all  }  z_i \in B^{d_i},\\
	      \tilde C_0|x_i-y_i|,    & \hbox{for all  }  z_i \in \R^{d_i}\setminus B^{d_i}.
	    \end{cases}$$
\item[(M5)] (\textit{Nondegeneracy}) $\;$ 
  There exists a constant $C_\mu>0$ and for $i=1,2$, there exist 
  $\beta_i\in(1,2)$ such that, for every $p_i\in\R^{d_i}$, 
  there exist $0<\eta_i<1$ such that the following holds for any
  $x_i\in\R^{d_i}$ and $\delta >0$
  $$ \int_{\mathcal C^i_{\eta_i,\delta}(p_i)}|j_i(x_i,z_i)|^2 d\mu_i(z_i)\geq C_\mu
\ 
\eta_i^{\frac{d_i-1}{2}}\ \delta^{2-\beta_i},$$
with $$\mathcal C^i_{\eta_i,\delta}(p_i):= \big\{z_i \in \R^{d_i}\ ; |j_i
(x_i,z_i)|\leq\delta, \space\
  (1-\eta_i)|j_i(x_i,z_i)||p_i|\leq|p_i\cdot j_i(x_i,z_i)|\big\}.$$
\end{itemize}
We point out that these assumptions on the measures and jumps are
either classical or used in \cite{bi08,bcci12,c12} to obtain
uniqueness, regularity results and Strong Maximum/Comparison
Principle.\medskip

Now we turn to the assumptions on the nonlinearities $F_i$, $i=1,2$,
$H$ and the source term $f$.  Since these assumptions are the same for
$F_1,F_2$, we write them for a general $F$ and in $\R^{\tilde d}$,
having in mind that they hold for $i=1,2$, $F=F_i$ with $\tilde d =
d_i$. We denote by $\S_{\tilde d}$ the space of $\tilde d \times
\tilde d$ symmetric matrices.
\begin{itemize}
 \item[(F0)] (\textit{Homogeneity}) For any $\lambda>0$, $x,p\in \R^{\tilde d},X \in \S_{\tilde d},l\in \R$, we have
$$F(x, \lambda
   p,\lambda X,\lambda l)= \lambda F(x,p,X,l).\ $$
 \item[(F1)] (\textit{Periodicity-Continuity}) $f$ and $x\mapsto
   F(x,\cdot,\cdot,\cdot)$ are continuous and
   $\Z^{\tilde{d}}$-periodic  in $\R^d$.
\item [(F2)] (\textit{Ellipticity-Growth conditions}) There exist two
  bounded functions ${\Lambda_1},{\Lambda_2}:\R^{\tilde d}\rightarrow
  [0,\infty)$ and a constant $\Lambda_0 >0$ such that
    ${\Lambda_1}(x)+{\Lambda_1}(x) \ge \Lambda_0$ and some constants
    $k\geq0$, $\tau\in(0,1]$, such
    that for any $x,y\in\R^{\tilde d}$, $p\in\R^{\tilde d}$, $l\leq
    l'$ and any $\varepsilon >0$
\begin{eqnarray*}&& 
F(y,p,Y,l')-F(x,p,X,l) \leq \\ && \hspace{2cm} 
{\Lambda_1}(x)\left((l-l') + \frac{|x-y|^2}{\varepsilon}+
|x-y|^\tau|p|^{k+\tau} + {{C_1}|p|^{k}}\right)+ \\ 
&& \hspace{2cm} {\Lambda_2}(x)\left( \tr(X-Y) + \frac{|x-y|^2}{\varepsilon} +
|x-y|^\tau|p|^{2+\tau}+{{C_2}|p|^2}\right)
\end{eqnarray*}
if $X,Y\in\S_{\tilde d}$ satisfy the inequality
\begin{equation}\label{eq::matrix_ineq}
- \frac{1}{\varepsilon}
\begin{bmatrix} 
I &  0 \\
0 & I 
\end{bmatrix} 
\leq
\begin{bmatrix} 
X &  0 \\
0 & -Y 
\end{bmatrix} 
\leq \frac{1}{\varepsilon}
\begin{bmatrix} 
 Z & -Z \\
-Z &  Z 
\end{bmatrix},
\end{equation}
with $ Z = I - \omega\,\hat p_0\otimes \hat p_0$, for some unit vector $\hat
p_0\in\R^{\tilde d}$, and {$\omega\in (1,2)$}.
\end{itemize}

\noindent As we mention it at the beginning of the section, (F2) does not seem to
be consistent with (F0), nor will be the next assumption (F3).  However,
we will comment the more general framework in Section
\ref{sec:extensions} dedicated to extensions. In the sequel, we use the
notations $\Lambda^{i}_1, \Lambda^{i}_2, k^{i}, \tau^{i}$ for the
quantities appearing in (F2) when they are related to $F_i$.

\begin{itemize}
\item [(F3)] (\textit{Lipschitz Continuity}) $(p,X,l)\mapsto F(x,p,X, l)$ is
  Lipschitz continuous, uniformly in $x\in \R^{\tilde d}$.
\item [(F4)] (\textit{Regularity}) There exists a modulus of
  continuity $\omega_F$ such that for any $\varepsilon >0$
\begin{eqnarray*}
 F(y,\frac{x-y}{\varepsilon},Y,l) - F(x,\frac{x-y}{\varepsilon},X,l)
 \leq \omega_{F}\left(\frac{|x-y|^2}{\varepsilon} + |x-y|\right)
\end{eqnarray*}
for all $x,y\in \R^{\tilde d}$, $X, Y \in \S_{\tilde d}$ satisfying
the matrix inequality (\ref{eq::matrix_ineq}) with $Z= I$ and
$l\in\R$.
\end{itemize}

\noindent Finally, on the Hamiltonian $H$, we assume one of the two
following hypotheses.
\begin{itemize}
 \item [(H-a)] (\textit{Sublinearity}) $H$ is locally Lipschitz
   continuous and there exists a Hamiltonian $\overline{H}(p)$,
   $1$-positively homogeneous such that
\[\lim_{k\rightarrow\infty} \frac1kH(kp) = \overline{H}(p)\; .\]
 \item [(H-b)] (\textit{Superlinearity}) $H$ is locally Lipschitz
   continuous and there exists $m > 1$, $\eta >0$, $r_0>0$ and
   $0<\mu_0<1$ such that, for all $\mu \in [\mu_0,1]$ and $|p|\geq
   r_0$
   \[ \mu H(\frac{p}{\mu})-H(p)\geq \eta (1-\mu)|p|^m\; .\]
\end{itemize}
We also use below a consequence of (H-b), namely the fact that there
exists $\hat \eta>0$, $c_0 \ge 0$ such that, for all $p \in \R^d$ and
all $c \ge c_0$,
\begin{equation}\label{propH}
c^{-1} H(cp)-H(p) \ge  \hat\eta c^{m-1} |p|^m - (\hat \eta)^{-1}\; .
\end{equation}
We leave the proof of (\ref{propH}) to the reader: for $|p|\geq r_0$,
it comes from (H-b) while, for $|p|\leq r_0$, it can be deduced from the
first case, taking $\hat \eta$ small enough.

\medskip

\paragraph{\bf Comments on the list of assumptions}

We need a long list of assumptions in order to apply known results
about uniqueness, regularity, Strong Maximum/Compa\-rison Principle
\textit{etc.} for different equations. In order to convince the reader
that we can indeed apply all these theorems, we next make a precise
list of the ones we will use and we justify that our assumptions imply
theirs.

To fit the framework of viscosity solutions (to ensure the existence
of continuous solutions when combined with Perron's method, but not
only), we will be using the Comparison Principle for both the
evolution equation~\eqref{eq:main} and for the perturbed stationary
equation~\eqref{eq:approx-ergo} introduced in the next
section. Comparison principle has been shown in \cite{bi08} to hold if
a series of assumptions (A1)-(A4) were satisfied. In our case (A1) in
\cite{bi08} comes from (M1), (M3), (M4) in the present paper; (A2) is
trivially satisfied; (A3-2) in \cite{bi08} comes from (F4) in the
present paper; (A4) in \cite{bi08} comes from (F3) in the present
paper.

Both for uniqueness of the solution for the ergodic problem
corresponding to equation ~\eqref{eq:ergo} and for establishing the
long time behavior of solutions of equation ~\eqref{eq:main} we will
be using Strong Comparison Principle of Lipschitz sub- and
supersolutions from \cite[Theorem~32]{c12}; (H) in \cite{c12} comes
from (F2), (F3) and (H-a)/(H-b) in the present paper. We would like to
point out that, from a rigorous point of view, (F3) does not yield (H)
in \cite{c12}. However, in view of the proof of this result, see the
very end of it, it is clear that (F3) is enough to conclude.

As mentioned in the introduction, we would like to deal with Lipschitz
solutions for the some of the equations that will appear in the
proof. Regularity of solutions for Eq.~\eqref{eq:tvd}, and in
particular Lipschitz estimates, follows from
\cite[Corollary~7]{bcci12}. Indeed, take $F_0 = \delta u + H(Du)$
which satisfies (H0) and (H2) from \cite{bcci12}; (H1) for $F_1$ and
$F_2$ in \cite{bcci12} follows from (F2) in the present paper; (H2) in
\cite{bcci12} follows from (F3) in the present paper; (H3) in
\cite{bcci12} follows from (F4) in the present paper; (J1)-(J5) in
\cite{bcci12} follows from (M1)-(M5) in the present paper. We also
need Lipschitz estimates given by \cite[Corollary~7]{bcci12} for
solutions of equation ~\eqref{eq:w-delta}in the sublinear case.
Choose $F_0 = \delta u$, and replace $F_1 (x,p,X,l)$ with
$F_1(p,X,l)+c^{-1}H(cp)$. Then $F_0$ trivially satisfies (H0) and (H2)
from \cite{bcci12}; (H1) for $F_1$ in \cite{bcci12} follows from (F3)
and (H-a) in the present paper; (M3)-(J4) in \cite{bcci12} follows
from (M3)-(M4) in the present paper.

Finally, to solve the ergodic problem in the sublinear case, we will
make usage of the Strong Maximum/Comparison Principle from \cite[Theorem~20]{c12}
for Eq.~\eqref{eq:homog}. This equation is degenerate elliptic and
nonlinearities are continuous, i.e. it satisfies (E) from \cite{c12};
moreover, it is $1$-homogeneous thanks to (F0) and (H-a); in
particular, it satisfies the scaling assumption (S) of
\cite{c12}. Thanks to (M5) and (F3), it also satisfies
$\mathrm{(N_{LI})}$.

\section{The stationary ergodic problem} \label{sec:ergodic}

In this section we discuss the solvability of the stationary ergodic
problem.  For the sake of simplicity we write below
\begin{multline*} F(x,Dv, D^2v, \mathcal{I}[v])\\=
 F_1(x_1,D_{x_1}v, D_{x_1x_1}^2v, \mathcal I_{x_1}[x,v])
+F_2(x_2,D_{x_2}v, D_{x_2x_2}^2v, \mathcal I_{x_2}[x,v]) \; .\end{multline*}

Our result is the following. 
\begin{theorem}\label{thm:ergo} Assume that
  (M1)-(M5), (F0)-(F4) with $k_i \leq \beta_i$, and either (H-a) or
    (H-b) holds.  There exists a unique constant $\lambda \in \R$ for
    which the stationary ergodic problem
 \begin{equation}\label{eq:ergo}
 F(x,Dv,D^2v, \mathcal{I}[v]) + H(Dv) = f(x) - \lambda.
\end{equation}
has a Lipschitz continuous periodic viscosity solution $v:\R^d \to
\R$. Moreover, $v$ is the unique Lipschitz continuous solution of
\eqref{eq:ergo}, up to an additive constant.
\end{theorem}
\begin{proof}
For any $\delta >0$, we consider as in \cite{lpv} the following approximate
equation
\begin{equation}\label{eq:approx-ergo}
\delta v^\delta + F(x,Dv^\delta,D^2v^\delta,\mathcal I[v^\delta])
+H(Dv^\delta) = f(x) \text{ in } \R^d.
\end{equation}
If $ M = ||F(\cdot, 0, 0, 0)||_\infty + |H(0)| + ||f||_\infty$, we notice
that $-\delta^{-1}M$ and $\delta^{-1}M$ are respectively sub- and
supersolutions of the above approximated equation. Then it follows
from Perron's method for integro-differential equations as described
for instance in \cite{i05}, and from the comparison principle
(Theorem~3 of \cite{bi08}) that there exists a unique bounded
viscosity solution $v^\delta$ which satisfies
\begin{equation}\label{eq:unif-bound}
-\frac{M}\delta \leq v^\delta(x) \leq \frac{M}\delta, \quad \hbox{for all  }  
x\in\R^d.
\end{equation}
Assumptions (M1),(M3),(M4) and (F4) on $F_i$, $j_i$ and $\mu_i$
for $i=1,2$ imply that there is a comparison result for
\eqref{eq:approx-ergo}.

By the periodicity of $F$ and $f$, $v^\delta(\cdot)$ and $v^\delta(\cdot+z)$ are
 both solutions of the above problem, for any $z\in \Z^d$. Then the uniqueness
of the solution implies that they are equal;
hence, $v^\delta$ is $\Z^d$-periodic.

We next consider $\tilde{v}^\delta(\cdot) = v^\delta(\cdot) - v^\delta(0)$.
The following proposition states the uniform boundedness of this
sequence of normalized functions. It is the crucial technical part of the
analysis of the ergodic problem.

\begin{proposition}\label{prop:key}
The sequence $\{ \tilde{v}^\delta \}_\delta$ is uniformly bounded. 
\end{proposition}
The proof of the proposition is postponed.  We remark that
$\tilde{v}^\delta$ satisfies
\begin{equation}\label{eq:tvd}
 \delta \tilde{v}^\delta + F (x,D\tilde{v}^\delta,D^2\tilde{v}^\delta, \I
   [\tilde{v}^\delta]) + H(D\tilde{v}^\delta ) = 
f (x) - \delta v^\delta (0), \quad x \in \R^d.
\end{equation}
We derive from Proposition~\ref{prop:key} and results from
\cite{bcci12} that $\{\tilde{v}^\delta \}_\delta$ is also
equi-Lipschitz continuous (at this point, the whole set of
assumptions is required).  Therefore, we can use Ascoli's Theorem to
extract a subsequence $\tilde{v}^{\delta_n}$ which converges locally
uniformly (and therefore uniformly, because of the periodicity) to a
Lipschitz continuous $\Z^d$-periodic function $\uinf$.  On the other
hand, \eqref{eq:unif-bound} implies that $\delta_n v^{\delta_n} (0)$
is bounded; hence, up to extracting again a subsequence, we can
further assume that $\delta_n v^{\delta_n} (0) \to \lambda$. By the
(continuous) stability of viscosity solutions of integro-differential
equations, see e.g. \cite{bi08}, we conclude that $\uinf$ is a
solution of \eqref{eq:ergo}.

Next we consider two solutions $(\lambda_i,v_i)$, $i=1,2$, of the ergodic
problem
\eqref{eq:ergo}. Then $u_i (x,t) = \lambda_i t + v_i (x)$ are two
solutions of \eqref{eq:main} with the initial condition
\[ u_i(x,0) = v_i(x) \quad x \in \R^d.\]
From the comparison principle for
\eqref{eq:main}, we conclude that, for all $t>0$,
\[ u_i (x,t) \le u_j (x,t) + ||v_i - v_j||_\infty \; .\]
Dividing by $t>0$ and letting $t\to +\infty$, this implies $\lambda_i \le
\lambda_j$. Since $i$ and $j$ are arbitrary, we conclude that
$\lambda_1=\lambda_2$. 

Finally, thanks to (F3), (H-a)/(H-b) and the Lipschitz continuity of
$v_1$ and $v_2$, we can apply the Strong Comparison Principle from
\cite[Theorem~32]{c12} to the Lipschitz solutions $u_1$ and $u_2$ of
\eqref{eq:main} and conclude that $u_1-u_2 = v_1 -v_2$ is constant. It
is worth pointing out that this step uses in a crucial way the Lipschitz
continuity of $v_1$ and $v_2$ because of the linearization procedure
(but which has to be used only in the superlinear case). This completes
the proof of Theorem~\ref{thm:ergo}. 
\end{proof}

\begin{remark}
Solving the stationary ergodic problem is still possible for some
cases when $\beta<1$, as we will point out in some examples, see
Section~\ref{sec:ex}. As a matter of fact, the solutions $\delta
v^\delta$ would be H\"older continuous and the stationary ergodic
problem would have a H\"older continuous solution. In this case, if
the uniqueness of the ergodic constant remains true, it is not clear
anymore that the solutions of the ergodic problem are unique up to an
additive constant. However, it is worth pointing out that the
Lipschitz continuity of the solutions of the ergodic problem is needed
(in general) in order to prove the asymptotic result for the PIDE.
\end{remark} 
 We now turn to the proof of the proposition. We distinguish the
 sublinear and the superlinear case.
\begin{proof}[Proof of Proposition~\ref{prop:key} in the
    sublinear case]
We argue by contradiction: we assume that
we can find a subsequence, that we still denote by $(\tilde
v^\delta)_{\delta}$, for which the associated sequence of norms blows up, i.e.
$$c_\delta:=||\tilde v^\delta||_\infty\rightarrow \infty \quad \text{ as }
\delta\rightarrow 0.$$
We next consider
$$w^\delta(x) = \frac{\tilde v^\delta(x)}{c_\delta}$$
which,  by (F0), satisfies
\begin{equation}\label{eq:w-delta}
\delta w^\delta  +  F(x, D w^\delta, D^2w^\delta, \mathcal I[w^\delta]) +
\frac{1}{c_\delta}H( c_\delta D w^\delta)=\frac{ f(x)
-\delta v_\delta(0)}{c_\delta}.
\end{equation}
Since $||w^\delta||_\infty = 1$ for all $\delta>0$,
$\{w^\delta\}_\delta$ is equi-Lipschitz continuous by the results of
\cite{bcci12}. Hence, we can extract a locally uniformly converging
subsequence (hence globally by periodicity); we denote by $w$ the
limit which is a $\Z^d$-periodic, Lipschitz continuous function with
$||w||_\infty=1$.

By the (continuous) stability result for viscosity
solutions, see e.g. \cite{bi08}, Assumption~(H-a) implies that $w$ is
a solution of
\begin{equation}\label{eq:homog}
F(x, Dw, D^2w, \mathcal I[w]) + \overline H(Dw) = 0, \quad \hbox{in  }\R^d.
\end{equation}
The limiting equation \eqref{eq:homog} is $1$-homogeneous; in
particular, it satisfies the scaling assumption (S) of \cite{c12}. If
we can check that it satisfies $\mathrm{(N_{LI})}$ from \cite{c12},
then this will imply that the equation enjoys the Strong Maximum
Principle (cf. \cite[Theorem~22]{c12}).

This gives the contradiction: indeed, on one hand, we know that
$||w||_\infty= 1$ and, by the continuity and periodicity of $w$, its
maximum/minimum value is attained and equal to $\pm1$. Henceforth, by
the Strong Maximum Principle, the function must be constant equal to
$\pm1$.  On the other hand $w(0) = 0$ since $w^\delta(0)=0$ for all
$\delta$, which is the desired contradiction.

We now show that indeed, assumption~$\mathrm{(N_{LI})}$ of \cite{c12}
is satisfied, and that it results from from (F2) and (M5).  Fix
$R_0>0$; we must check that for all $R \in (0,R_0)$ and for all $c>0$
we have that
\begin{multline*}
 \mathcal A (\gamma):=\sum_{i=1,2} F_i\left(x_i,p_i, I-\gamma
 p_i\otimes p_i, \tilde C_\mu - c \gamma
 \int_{\tilde{\mathcal{C}}^i_{\eta_i,\gamma}(p_i)} |p_i\cdot
 j_i(x_i,z_i)|^2 d\mu(z_i)\right) \\ + \bar H(p) \rightarrow \infty,
\end{multline*} 
as $\gamma \to \infty$, uniformly for $x \in \R^d$, $p \in B_R \setminus
 B_{R/2}$. Here the constant $\tilde C_\mu$ is given by (M1) and the
 cone is slightly different from the one in (M5). Namely, it has the
 form
\[\tilde{\mathcal{C}}_{\eta,\gamma}(p)= \{z; \space
(1-\eta)|p||j(x,z)|\leq|p\cdot j(x,z)|\leq 1/\gamma\}.\] Going back to
the original form of the nonlinearity $F = F_1 + F_2 + H$ and using
the ellipticity - growth assumption (F2) we get the following
lower bound for $\mathcal{A} (\gamma)$, 
\begin{multline*}
 \mathcal A (\gamma) \ge -\sum_{i=1,2} \Big(\Lambda^i_1(x) \big(\tilde C_\mu -c\gamma\int_{\tilde{\mathcal{C}}^i_{\eta_i,\gamma}(p_i)} |p_i\cdot 
j_i(x_i,z_i)|^2d\mu(z_i)\big)
 + \Lambda^i_1(x) C_1^i|p_i|^{k_i} +\\
 \Lambda^i_2(x)\tr(I_{d_i}-\gamma p_i \otimes p_i)
 + \Lambda^i_2(x) C_2^i |p_i|^2 \Big)\\ 
 +  F_1(x_1,p_1,0,0) + F_2(x_2,p_2,0,0) + \bar H(p).
\end{multline*}
 Using now the structure of the cone, and noting that 
$ \mathcal{C}_{\eta_i, (R\gamma)^{-1}}(p) \subset
\tilde{\mathcal{C}}_{\eta,\gamma} (p)$, we get
\begin{multline*}
 \mathcal A (\gamma) \ge  -\sum_{i=1,2} \Big(\Lambda^i_1(x)
  \big(\tilde C_\mu - c\gamma(1-\eta_i)^2
|p_i|^2\int_{\mathcal{C}^i_{\eta,(R\gamma)^{-1}}(p_i)}
|j_i(x_i,z_i)|^2d\mu(z_i)\big)
 +\\
  \Lambda^i_2(x)(d_i -\gamma
|p_i|^2) \Big) + \mathcal O_\gamma(1).
\end{multline*}
Employing now the nondegeneracy assumption (M5)  and setting by $$C:=c C_\mu \min_{i=1,2}((1-\eta_i)^2 \eta_i^{\frac{d_i-1}2}
)R^{\beta_i-2}$$ we further have that 
\begin{align*}
\mathcal A (\gamma) \ge & 
\sum_{i=1,2} \big(\Lambda^i_1(x) C|p_i|^2 \gamma^{\beta_i-1}  +
\Lambda^i_2(x)|p_i|^2  \gamma \big) + \mathcal O_\gamma(1)\\
 \ge & 
\sum_{i=1,2} C \Lambda_0 |p_i|^2 \gamma^{\beta_i-1} + \mathcal
O_\gamma(1)\\
= & C \Lambda_0 |p|^2 \gamma^{\beta_i-1} + \mathcal O_\gamma(1)
\end{align*}
where $\mathcal O_\gamma(1)$ is bounded when $\gamma \to
\infty$. Therefore $ \mathcal A (\gamma) \rightarrow \infty$ as
$\gamma\rightarrow\infty$ uniformly in $x$, and $R/2<|p|<R$ and the
proof is now complete in the sublinear case.
\end{proof}
\begin{proof}[Proof of Proposition~\ref{prop:key} in the
    superlinear case.]  The beginning of the proof in the superlinear
  case goes along the same lines as in the sublinear case.  We assume
  that $c_{\delta}\to\infty$ along a subsequence and reach a
  contradiction.  For simplicity we still keep the notation $\delta$
  for the subsequence and assume (with no restriction) that
  $c_\delta\geq1$ for all $\delta>0$.  We consider as before the
  rescaled functions
\[w^\delta(x):=\frac{\tilde v^\delta(x)}{c_\delta},\] 
so that $\|w^\delta\|_\infty=1$ for any $\delta>0$. Rewriting the
equation with $w^\delta$ and using again the 1-positive homogeneity of
nonlinearities, cf. (F0), we see that $w^\delta$ still satisfies
\eqref{eq:w-delta}. In particular, we get from (H-b) (and more precisely from (\ref{propH})) that
\begin{multline*}
 \hat \eta|Dw^\delta|^m \leq - \frac{1}{c_\delta^{m-1}} (H (Dw^\delta)-(\hat\eta)^{-1})
  -\frac{1}{c_\delta^{m-1}}F(x, Dw^\delta, D^2w^\delta,  \I
  [x,w^\delta]) \\
  -\frac{\delta w^\delta}{c_\delta^{m-1}}+\frac{f(x)-\delta
v^\delta(0)}{c_\delta^m}.
\end{multline*}
We next claim that the following holds true.
\begin{lemma}\label{lemma:unif_cont}
    The family $\{w^\delta\}_{\delta>0}$ is equicontinuous in $\R^d$.
\end{lemma}
The proof of the lemma is postponed and we complete the proof of the
proposition. Since $\|w_\delta\|_\infty \le 1$, using Ascoli's
Theorem, we can extract a subsequence which converges locally
uniformly (hence globally) towards a continuous $\Z^d$-periodic
function $w$. Using standard stability results for viscosity solutions
together with Estimate~\eqref{eq:unif-bound}, $c_\delta \to \infty$
and (F4), we get
\[\hat \eta|Dw|^m \leq0\,.\]
From this we deduce that $Dw=0$ in the viscosity sense. But there we get a
contradiction from the continuity of $w$ since $w(0)=0$ and
$\|w\|_\infty=1$.  The proof of the proposition is now complete.
\end{proof}
\begin{proof}[Proof of Lemma \ref{lemma:unif_cont}]
We first fix a parameter $\alpha\in(0,1)$ and claim that
\begin{claim}
  For any $\mu\in(0,1)$, there exists a (large) constant $L(\mu)>0$
  such that for any $x,y\in \R^d$, we have
\[w^\delta(x)-w^\delta(y)\leq L(\mu)|x-y|^\alpha+2(1-\mu)\,.\]
\end{claim}
It is classical that such a result implies that $w^\delta$ is uniformly
continuous and that the modulus of continuity only depends on $L$ and
$\alpha$. To prove the claim, we consider the function
\[\Phi(x,y):=\mu w^\delta(x)-w^\delta(y)-L|x-y|^\alpha\,.\]
Since $w^\delta$ is $\Z^d$-periodic, this function reaches its maximum
at some point $(\bar x,\bar y)$ that we may consider in the same
cell. If this maximum is nonpositive, then we are done since this
means that $w^\delta(x)-w^\delta(y)\leq L|x-y|^\alpha$ for any $x,y\in
\R^d$.

Otherwise there are two options: $(i)$ either $\bar{x}=\bar{y}$; $(ii)$ or
$\bar{x}\neq\bar{y}$.

In case $(i)$ we have for any $x,y$
\[\mu w^\delta(x)-w^\delta(y)-L|x-y|^\alpha\leq (\mu-1)w^\delta(\bar x)\,,\]
which implies, together with $\|w^\delta\|_\infty\leq1$, that
\[w^\delta(x)-w^\delta(y)\leq L|x-y|^\alpha + {2}(1-\mu)\,,\]
and the claim holds. 

In case $(ii)$, we can use the viscosity inequalities since the
functions $L|\cdot-\bar y|^\alpha$ and $L|\bar x-\cdot|^\alpha$ are
smooth near $\bar x$ and $\bar y$ respectively.  Notice that
$\max\Phi\geq\Phi(\bar x,\bar x)={(\mu-1)}w^\delta(\bar x)$ which
gives the estimate
\begin{equation}\label{est:max.point.1}
  |\bar x-\bar y|\leq \Big(\frac{2}{L}\Big)^{1/\alpha}\,.
\end{equation}

The equation for $\mu w^\delta$ can be derived from \eqref{eq:w-delta}
\[
   F(\bar x, D (\mu w^\delta), D^2 (\mu w^\delta),
   \I [\bar x, \mu w^\delta]) +
  \frac{H\big((c_{\delta}/\mu) D(\mu \tilde w^\delta)\big)}{c_\delta/\mu}
  +\delta \mu w^\delta=\frac{f(x)-{\delta} v^\delta(0)}{c_\delta/\mu}\,.
\]
We use the nonlocal version of Jensen-Ishii's Lemma (cf. \cite{bi08}) and
computations from \cite[p.14]{bci11} (see Step~2 of the proof of
Theorem~3.1 of \cite{bci11}) in order to get for all $\delta>0$ two
matrices $X,Y \in \S_d$ such that
\begin{align*}
  F (\bar x, P, X, \I [\bar x, \mu w^\delta]) +
  \frac{H\big((c_\delta/\mu)P \big)}{c_\delta/\mu}+\delta\mu
  w^\delta(\bar x) &\leq \frac{f(\bar x)-\delta
    v^\delta(0)}{c_\delta/\mu} \\ F (\bar y, P, Y, \I [\bar y,
    w^\delta] ) +\frac{H\big(c_\delta P
    \big)}{c_\delta}+\delta w^\delta(\bar y) &\geq \frac{f(\bar
    y)-\delta v^\delta(0)}{c_\delta}
\end{align*}
where 
\[P=\alpha L|\bar x -\bar y|^{\alpha-2}(\bar{x} - \bar{y})\]
and $X,Y \in \S_d$ satisfy
\[
\begin{bmatrix} X & 0 \\ 0 & -Y \end{bmatrix} \le
\frac2{\bar \eps} 
\begin{bmatrix} Z & -Z \\ -Z & Z \end{bmatrix} 
\le \frac6{\bar \eps} 
\begin{bmatrix} I & -I \\ -I & I \end{bmatrix} 
\]
where $Z=I - (1+\varpi) \hat a \otimes \hat a$ with $a= \bar x -
\bar y$ and $\bar \eps = (L \alpha)^{-1} |\bar x - \bar y|^{2-\alpha}$
and $\varpi \in (0,1/3)$.
Now we subtract both inequalities and we obtain
\begin{equation}\label{eq:start}
  \frac{H\big((c_\delta/\mu)P\big)}{c_\delta/\mu} -
  \frac{H\big(c_\delta P\big)}{c_\delta} \le (1+\mu) \frac{
    \|f\|_\infty + M}{c_\delta} + (1+\mu)\delta + T_{\mathrm{h.o.t.}}
\end{equation}
where 
\[
T_{\mathrm{h.o.t.}} = F (\bar y, P, Y, \I [\bar y,
w^\delta]) - F (\bar x, P, X, \I [\bar x, \mu
w^\delta]).
\]
Using computations from \cite{bcci12}, we get the following estimate
of the higher order terms.
\begin{lemma}\label{lem:hot}
There exists $\tilde{c}$ depending on constants appearing in (M1)-(M5)
such that
\[\I [\bar x, \mu w^\delta] -\I [\bar y, w^\delta] \le \tilde{c}.\]
\end{lemma}
\begin{proof}
We use \cite[Corollary~16]{bcci12} with $u=\mu w^\delta$ and
$v=w^\delta$, $t_0=2\sqrt{d}$, $\gamma=1$. We then get the desired
estimate with $\tilde{c}= 2 (C^2_\mu+C^3_\mu)$ where $C^2_\mu$ and
$C^3_\mu$ depend on constants appearing in (M1)-(M5). For the precise
estimate of $\tilde{c}$ given above, see the proof of
\cite[Corollary~16]{bcci12}. The proof of Lemma~\ref{lem:hot} is
now complete.
\end{proof}
Combining (F3)-(F4) with Lemma~\ref{lem:hot}, we thus get
\[ T_{\mathrm{h.o.t.}} \le \omega_F (\eps^{-1} |\bar x -\bar y|^2 + |\bar x -
\bar y|) + C_F \tilde{c}\] 
where $C_F$ is the Lipschitz constant coming from (F3) and $\eps =
L^{-1} |\bar x -\bar y|^{2-\alpha}$. We therefore have
\[ T_{\mathrm{h.o.t.}} \le \omega_F (L|\bar x -\bar y|^\alpha + |\bar x -
\bar y|) + C_F \tilde{c} \leq \omega_F(2 + (2/L)^{1/\alpha}) +
C_F \tilde{c}\]

We now turn to first-order terms. We first notice that
(\ref{est:max.point.1}) yields
$$ |P|\geq 2^{\frac{\alpha-1}{\alpha}}\alpha L^{1/\alpha}\; ,$$ and
since we may assume without loss of generality that $c_\delta \geq 1$,
we have $|c_\delta P|\geq r_0$ for $L$ large enough. Therefore we can
use (H-b) which yields
\[\frac{H\big((c_\delta/\mu)P\big)}{c_\delta/\mu} -\frac{H\big(c_\delta P\big)}{c_\delta} \geq  \eta(1-\mu) c_\delta^{m-1} |P|^m\,.\]
Hence \eqref{eq:start} together with Lemma~\ref{lem:hot} yields, for $\delta$ small enough
\[\eta (1-\mu) c_\delta^{m-1} |P|^m  \leq  2(
    \|f\|_\infty + M +1) +  \omega_F(2 + (2/L)^{1/\alpha}) +
C_F \tilde{c}=:\hat c\; . \] Now, using the above estimate of $P$ we have
\[
  \eta c_\delta^{m-1}|P|^m  \geq c_0 c_\delta^{m-1}{}L^{m/\alpha}
\]
where $c_0= \eta \alpha^m 2^{m(\alpha-1/\alpha)}$. The above
inequality then becomes
\[c_0(1-\mu) \le c_0 c_\delta^{m-1} \leq \hat{c} L^{-m/\alpha}.\]
Hence, we reach a contradiction by choosing $L$ large enough
(depending on $\mu$).  The proof of the claim (and thus of the lemma)
is now complete.
\end{proof}

\section{Large Time Behavior}\label{sec:longtime}

This section is devoted to the proof of the following theorem. 
\begin{theorem}\label{thm:conv}
 For every $\Z^d$-periodic,  Lipschitz continuous initial data $u_0$,   
 the unique $\Z^d$-periodic solution $u \in C(\R^d\times (0,+\infty) )$ of
   \eqref{eq:main}-\eqref{eq:ic} satisfies 
\[
u(x,t)- \lambda t - v(x) \rightarrow 0, \quad \text{ as }
t\rightarrow\infty
\]
where $(\lambda,v)$ is a solution of the stationary ergodic problem
\eqref{eq:ergo}.
\end{theorem}
\begin{proof}[Proof of Theorem~\ref{thm:conv}.]
We first prove that, for any $\Z^d$-periodic, Lipschitz continuous
initial data $u_0$, Eq.~\eqref{eq:main}-\eqref{eq:ic} has a solution
$u$. In order to do so, we can assume that $u_0$ is $C^2$ and bounded,
with first and second derivatives bounded as well. Indeed, general
initial data $u_0$ can be handled thanks to regularization and
stability of viscosity solutions of \eqref{eq:main}, see
e.g. \cite{bi08}. For smooth $u_0$'s, we make the classical
observation that $u_0\pm Ct$ is a super-/subsolution of
\eqref{eq:main} for $C$ large enough. We thus can apply Perron's
method and get a discontinuous solution. The comparison principle for
\eqref{eq:main} (cf. \cite{bi08}) then implies that this solution is
continuous and unique.
\medskip

We next follow closely \cite{bs01}. However, we give details for
the reader's convenience. 
 The functions $u$ and $\lambda t+v(x)$ are both solutions in
$ \R^d\times (0,\infty)$ of \eqref{eq:main}. Hence, by the
comparison principle for \eqref{eq:main}, we have 
\begin{equation}\label{eq:lt-unif-bound}
 ||u(x,t) - v(x) -\lambda t||_\infty \leq ||u_0(x)-v(x)||_\infty
\end{equation}
and,  more generally, for all $t\geq s \geq 0$ 
\begin{equation}\label{eq:lt-monotone}
\max_{x\in\R^d}\big(u(x,t) - \lambda t - v(x)\big) \leq 
\max_{x\in\R^d}\big(u(x,s) - \lambda s - v(x)\big).
\end{equation}
Introducing the notation 
\[m(t) = \max_{x\in\R^d}\big(u(x,t)-\lambda t - v(x)\big),\]
we deduce from \eqref{eq:lt-unif-bound} and \eqref{eq:lt-monotone}
that $m(t)$ is nonincreasing and bounded. In particular, there exists
$\overline m \in \R$ such that 
\begin{equation}
m(t) \searrow \overline m \text{ as } t\rightarrow\infty.
\end{equation}
Define next 
\[ w(x,t) = u(x,t) - \lambda t.\]
From \eqref{eq:lt-unif-bound} we have  for all $t>0$,
\[
||w(\cdot,t)||_\infty \leq ||u_0-v||_\infty + ||v||_\infty.
\]
This $L^\infty$-bound implies that $\{w(\cdot,t)\}_{t>0}$ is equi-Lipschitz
continuous by the results of \cite{bcci12}.  In particular, the sequence
$\{w(\cdot,n)\}_{n \in \N}$ is compact in
$C(\R^d)$. We thus can extract a converging subsequence
$(w(\cdot,\phi(n)))_n$
\[ w(x,\phi(n)) \to \overline v (x), \quad x \in \R^d.\]
We next deduce from the comparison principle for \eqref{eq:main} that
$(w(\cdot,\phi(n)+\cdot))_n$ is a Cauchy sequence in $C(\R^d\times (0,\infty))$.
Indeed, for all $t>0$,
\[
||w(\cdot, t+\phi(n))-w(\cdot, t+\phi(m))||_\infty \leq ||w(\cdot, \phi(n))
- w(\cdot,\phi(m))||_\infty.
\]
We thus deduce that $w(x,\phi(n)+t)$ converges uniformly to
$\overline w \in C(\R^N\times(0,\infty))$ solving 
\begin{equation}\label{eq::PIDEs_lambda}
\left\{
  \begin{array}{ll}
     \medskip 
     \overline w_t + F(x,D\overline w,D^2\overline w,\mathcal{I}[\overline
w]) + H(D\overline w) + \lambda = f(x) & \hbox{ in } \R^d \times (0,+\infty)\\
     \overline w(x,0)  =  \overline v(x) & \hbox{ in } \R^d.   
  \end{array}
\right.
\end{equation}
In particular, $\overline w(\cdot,t)$ is Lipschitz continuous for all $t>0$. 
Passing to the limit  $\phi(n) \rightarrow\infty$ in 
\[m(\phi(n)+t) = \max_{x\in\R^d}\big(w(x,\phi(n)+t)-v(x)\big)\]
thanks to the uniform convergence of the sequence $(w(\cdot,
\phi(n)+\cdot))_n$, we obtain that for all $t>0$, 
\[\overline m  = \max_{x\in\R^d}\big(\overline w(x,t)-v(x)\big).\]
Thanks to (F3), (H2) and the Lipschitz continuity of $\overline
w(t,\cdot)$ and $v$, we can apply the Strong Comparison Principle from
\cite{c12}: for all $t>0$, $x \in \R^d$,
\[\overline w(x,t) = v(x) + \overline m.\]
This implies in particular that $\overline v =v+ \overline m$. Hence,
$v+\overline m$ is the only possible limit of $w(\cdot,n)$; hence
$w(x,n) \to v (x)+\overline m$ uniformly as $n \to \infty$. Now for $t \ge n$,
we have
$$
\|u(\cdot,t) - \lambda t - v - \overline m \|_\infty=
\|w(\cdot,t) - v -\overline m \|_\infty \le 
\|w(\cdot,n) - v -\overline m\|_\infty \to 0
$$ as $n \to \infty$.  Since the stationary ergodic equation is blind
to additive constants, we can replace $v$ with $v+\overline m$. The
proof is now complete.
\end{proof}

\section{Remarks and Extensions}\label{sec:extensions}

In this section, we explain how to study the large time asymptotic of
solutions of equations that do not satisfy the assumptions of our main
results.

The first key point concerns the {\em existence and uniqueness} of
viscosity solutions, both for the evolution equation and for the
approximate equations which are used to solve the ergodic problem. For
this part, (F4) plays a key role and imposes rather strong
requirements on the $F_i$'s, which a priori should reduce drastically
the growth possibilities in (F2). But one may turn around this
difficulty by using a truncation argument: truncating the gradient
terms in $F_1$ and $F_2$ leads to an equation where existence and
uniqueness of viscosity solutions holds by assuming only that (F4)
holds for bounded $p$'s, then one can use (F2) to show that the
solution is actually Lipschitz continuous (assuming that $u_0$ is
Lipschitz continuous or using a regularizing effect). In that way,
with a slight modification of the formulation of (F4) (which has only
to hold for bounded $p$'s), we reconcile Assumptions (F2) and (F4).

But can $F_1, F_2$ really be superlinear? This question is also
connected to the generalization of the $H$-term and we start by this
issue.

Following the framework of \cite{bcci12}, one should be able to
replace the term $H(Du)$ by $F_0(Du,D^2u,\I[u])$. We recall that, as
far as the Lipschitz regularity is concerned, the key point is that
this term is independent of $x$ but it can be ``degenerate''. This is
a ``neutral'' term which does not bring anything positive nor creates
any difficulty. The same is true for the comparison result.

For proving Theorem~\ref{thm:ergo}, we need the analogue of
(H-a)--(H-b), i.e. we have to examine $F_0(cp,c X,cI)$ as $c$ tends to
infinity. The analogous assumptions are 
\begin{itemize}
 \item [(F0-a)] (\textit{Sublinearity}) $F_0$ is locally Lipschitz
   continuous and there exists a nonlinearity $\overline{F_0}$,
   $1$-positively homogeneous such that
\[\lim_{c\rightarrow\infty} \frac1cF_0(cp,c X,cI) = \overline{F_0}(p,X,l)\; .\]
 \item [(F0-b)] (\textit{Superlinearity}) $F_0$ is locally Lipschitz
   continuous and there exist $m > 1$, $\eta >0$ such that, for any $c\geq 1$, for any $p\in \R^d, X\in \S_d,l\in \R$
\[\frac1cF_0(cp,cX,cI) -F_0(p,X,I) \geq  \eta (c^{m-1}-1) |p|^m -\eta^{-1}\; .\]
\end{itemize}
This last assumption is obviously satisfied if we add a
$1$-homogeneous function of $p,X,l$ to an $H$ for which of course the
superlinear case may come from $F_1,F_2$: typically in the case when
$F_0\equiv 0$, one may assume that (H-b) holds. On the other hand
(F0-b) summarizes the two types of properties we use to prove our
results.

We may also assume that $F_1,F_2$ satisfy (F0-b) with the same $m$:
in this case, the proof follows along the same lines, $|p|^m$ being
replaced by $|p_1|^m + |p_2|^m$.

\smallskip

Let us finally mention that our results can be naturally extended to
second order fully nonlinear parabolic integro-differential equations,
such as those appearing in stochastic control of jump processes. For
instance, we can consider the following \textit{Bellman-Isaacs
  Equations} in $\R^d$ 
\begin{eqnarray*}
\partial_t u + \sup_{\gamma\in\Gamma} \inf_{\delta\in \Delta}\Big( & - \frac12 \hbox{tr}(A^{\gamma,\delta}  (x) D^2 u)  - \mathcal J^{\gamma,\delta} [x,u] 
- b^{\gamma,\delta}(x) \cdot D u  - f^{\gamma,\delta}(x) &\\
 &
 - \frac12 \hbox{tr}(a^{\gamma,\delta} _1(x_1) D_{x_1x_1}^2 u)  - \mathcal J_{x_1}^{\gamma,\delta} [x,u] 
- b_1^{\gamma,\delta}(x) \cdot D_{x_1} u  & \\
&
- \frac12 \hbox{tr}(a^{\gamma,\delta} _2(x_2) D_{x_2x_2}^2 u)  - \mathcal J_{x_2}^{\gamma,\delta} [x,u] 
- b_2^{\gamma,\delta}(x) \cdot D_{x_2} u 
& \Big) =0
\end{eqnarray*}
where $\mathcal J^{\gamma,\delta} [x,u]$ is a family of L\'evy-It\^o
operators associated with a common L\'evy measure $\mu^0$ and a family
of jump functions $j_0^{\gamma,\delta}(x,z)$, respectively $\mathcal
J_{x_i}^{\gamma,\delta} [x,u]$ are families of L\'evy-It\^o operators
associated with the L\'evy measures $\mu^i$ and the families of jump
functions $j_i^{\gamma,\delta}(x_i,z)$, for $i=1,2$. We consider that
$A^{\gamma,\delta}, a^{\gamma,\delta} _i, b^{\gamma,\delta}_i,
f^{\gamma,\delta} $ are bounded in $W^{1,\infty}$, uniformly in
$\gamma$ and $\delta$.

\section{Examples}\label{sec:ex}

In this section we give several examples for which the previous
results (and their extensions from Section~\ref{sec:extensions})
apply.  We recall that proving this type of behavior required two main
tools: regularity of solutions, and Strong Maximum Principle (shortly
SMP) for the stationary equation rescaled as in equation
(\ref{eq:homog}), and Strong Comparison Principle (shortly SCP) for
the evolution equation.  We comment three classes of equations:
classical diffusions, composed local-nonlocal equations, and mixed
equations.

\subsection{Classical nonlocal diffusions} 

Our results apply to a large class of classical PIDE and it would be
difficult to illustrate this generality. Therefore, we rather present
two particular equations, where we point out some extensions of our
results to fractional exponents of lower order $\beta<1$, both in the
sublinear and superlinear case.

\begin{example}[\bf Fractional diffusions with drift]\label{ex:fr_drit}
The model example in this case is the following one
\begin{equation*} \left\{
\begin{array}{ll}\medskip
\partial_t u + (-\Delta)^\beta u + b(x) \cdot Du = f(x) & \hbox{ in } \R^d\times (0,\infty) \\
u(\cdot, 0) = u_0(\cdot) & \hbox{ in } \R^d
\end{array}\right.
\end{equation*}
where the vector field $b\in C^{0,\tau}(\R^d;\R^d)$ is $\Z^d$-periodic
and $u_0, f$ are continuous and $\Z^d$-periodic. The nonlocal term is
a fractional Laplacian of order $\beta\in(1-\tau,2)$.
$$
(-\Delta)^\beta u (x) = -\int_{z \in \R^{d}} \big( u(x + z) - u(x) - Du(x)\cdot
z 1_{B^{d}}(z)\big)\frac{dz}{|z|^{d+\beta}}\; .
$$ 

Then, for H\"older continuous initial data, the solution is H\"older
continuous only if $\beta > 1-\tau$. The interesting news in this case
is that the corresponding stationary ergodic problem still satisfies
SMP, even if $\beta<1$ (by translations of measure supports, see
\cite{c12}). Hence, by similar arguments, Theorems \ref{thm:ergo} and
\ref{thm:conv} hold for fractional exponents if $ 1-\tau
  < \beta< 2$.

Therefore, the equation has a H\"older continuous viscosity solution
$u(x,t)$ which for large times behaves like $\lambda t + v(x)$, with
$(\lambda,v)\in\R\times C^{0,\alpha}(\R^d)$ an ergodic pair for
\[(-\Delta)^\beta v + b(x) \cdot Dv = f(x) -\lambda \hbox{ in }\R^d.\]
\end{example}

\begin{example}[\bf Superlinear Equations]\label{ex:superlin}
Consider the equation
\begin{equation*} \left\{
\begin{array}{ll}\medskip
\partial_t u + (-\Delta)^\beta u + b(x)\vert Du \vert  + |Du|^m= f(x) & \hbox{ in } \R^d\times (0,\infty) \\
u(\cdot, 0) = u_0(\cdot) & \hbox{ in } \R^d.
\end{array}\right.
\end{equation*}
with $u_0,f,b\in C^{0,\tau}(\R^d)$, $\beta\in(1-\tau,2)$, and $m\geq
0$. Here we point out that we may allow $m$ to be less than $1$, even
if the nonlinearity is not locally Lipschitz continuous because in the
linearization process we are using, we can use the H\"older continuity
of $p \mapsto |p|^m$ instead of its local Lipschitz continuity.

Solutions are Lipschitz continuous for Lipschitz initial data, for
all $\beta>1$, and $\tau-$ H\"older continuous for $C^{0,\tau}$
initial data, when $1-\tau<\beta\leq 1$. To establish the long time
behavior for any $\beta\in(1-\tau,2)$ we use as before SMP in the
sublinear case $m\leq1$, and Lemma \ref{lemma:unif_cont} in the
superlinear case $m>1$.
\end{example}

\subsection{Composed local-nonlocal diffusion} 

We refer in the sequel examples to equations which are strictly
elliptic in a generalized sense, as introduced in \cite{bci11}: at
each point the nonlinearity is either non-degenerate in the second
order term, or is non-degenerate in the nonlocal term. One can imagine
a partition of the periodic cell in two components such that the local
diffusion vanishes on one of them and there the nonlocal term is uniformly
elliptic, and vice-versa on the other component.

\begin{example}\label{ex:composed}
The model example in this case is the following one
\begin{equation*} \left\{
\begin{array}{ll}\medskip
\partial_t u - a_1(x)\Delta u + a_2(x)(-\Delta)^\beta u + |Du|^m = f(x) & \hbox{ in } \R^d\times (0,\infty) \\
u(\cdot, 0) = u_0(\cdot) & \hbox{ in } \R^d.
\end{array}\right.
\end{equation*}
with $a_1(\cdot)= \sigma^2(\cdot), a_2(\cdot)\geq 0$, and where
$\sigma,a_2, f, u_0$ are Lipschitz continuous and $\Z^d$-periodic. We
take $\beta\in(1,2)$, and $m\geq 1$.

Solutions for this equation are Lipschitz continuous when $\beta>1$,
provided the two diffusions do not cancel simultaneously
\begin{equation}\label{eq:deg}
a_1(x) + a_2(x) \geq a_0 >0.
\end{equation}
Also, SMP holds (by translations of measure supports) for all
$\beta>1$ even if one of the two local or nonlocal diffusion
completely degenerates.

However, in order to prove existence of solutions as well as to use
SCP, the equation must satisfy hypotheses which ensure
comparison/uniqueness results. For this reason, we can either take $
a_2(x) \equiv a_2 >0 \hbox{ for all } x\in\R^d,$ or we need to impose
that $a_2(\cdot)^{1/ \beta}$ is Lipschitz continuous, in which case
$a_2(\cdot)$ can vanish. We emphasize on the fact that in the latter
case, we can deal with degenerate nonlocal diffusions and hence work
under assumption (\ref{eq:deg}).

Indeed, when $a_2(x)>0$ performing the change of variables
$z=\alpha(x)z'$, with $\alpha(x) = a_2(x)^{1/\beta}$ the nonlocal
diffusion can be re-written as
\begin{eqnarray*}
[a_2 (-\Delta)^\beta u] (x) &~\hspace{-3mm}=& ~\hspace{-4mm}
\int_{\R^{d}} \big( u(x + z) - u(x) - Du(x)\cdot
z 1_{B^{d}}(z)\big)\frac{a_2(x) dz}{|z|^{d+\beta}}\\
 &~\hspace{-3mm} = & ~\hspace{-4mm}
\int_{\R^{d}} \big( u(x + \alpha(x)z) - u(x) - Du(x)\cdot
\alpha(x)z 1_{B^{d}}(z)\big)\frac{a_2(x)\alpha(x)^d dz}{\alpha(x)^{d+\beta}|z|^{d+\beta}}. 
\end{eqnarray*}
Hence we get a L\'evy-It\^o operator with jump function  $j(x,z) = \alpha(x)z = a_2^{1/\beta}(x)z$

\[ a_2(x) (-\Delta)^\beta u (x)
 =
\int_{\R^{d}} \Big( u(x + a_2^{1/\beta}(x)z) - u(x) - Du(x)\cdot
a_2^{1/\beta}(x)z 1_{B^{d}}(z)\Big)\frac{dz}{|z|^{d+\beta}}.\]
Note that when $a_2(x)=0$ for some $x\in \R^d$ the above identity remains true.

This nonlocal operator satisfies the hypothesis $(A1)$ of \cite{bi08}: the first three integral bounds in $(A1)$ are immediately satisfied, whereas the last one is true provided that $\beta>1$ and $a_2(\cdot)^{1/\beta}$ is Lipschitz
\begin{eqnarray*}
 \int_{\R^d\setminus B}\big|j(x,z)-j(y,z)\big|\frac{dz}{|z|^{d+\beta}}
& =  &
\int_{\R^d\setminus B}\big|a_2(x)^{1/\beta}-a_2(y)^{1/\beta}\big||z|\frac{dz}{|z|^{d+\beta}}\\
& \leq &
\bar c_0 |x-y|\int_{\R^d\setminus B}|z|\frac{dz}{|z|^{d+\beta}} = \bar c |x-y|.
\end{eqnarray*}
Under this condition, we have the expected asymptotic behavior.
\end{example}

\subsection{Mixed local-nonlocal equations} 

These equations correspond to the so called \emph{mixed ellipticity:}
at each point the nonlinearity is degenerate both in the second order
term and the nonlocal term: classical diffusion makes the equation
uniformly elliptic in one direction only, say $x_1$, and fractional
diffusion makes it uniformly elliptic in the complementary direction
$x_2$. The interesting part is that this combination renders the
nonlinearity uniformly elliptic overall. In \cite{c12} these equations
were proven to satisfy the Strong Maximum and Comparison Principles,
and in \cite{bcci12} regularity results were established.  We comment
below how these two ingredients combine and give the large time
behavior. We first consider a toy-model and then a more general
equation, so that we can comment precisely the large set of
assumptions presented in Section ~\ref{sec:assumptions}.

\begin{example}[\bf Toy model]\label{ex:toy}
The model example in this case is the following one
\begin{equation*} \left\{
\begin{array}{ll}\medskip
\partial_t u - \Delta_{x_1} u + (-\Delta)_{x_2}^\beta u + |Du|^m = f(x) & \hbox{ in } \R^d\times (0,\infty) \\
u(\cdot, 0) = u_0(\cdot) & \hbox{ in } \R^d.
\end{array}\right.
\end{equation*}
where $u_0,f:\R^d\rightarrow \R$ are Lipschitz continuous and
$\Z^d$-periodic, $1<\beta<2, m\ge 0$. As before, we point out that we
may allow $m$ to be less than $1$, even if the nonlinearity is not
locally Lipschitz continuous, because in the linearization process we
are using, we can use the H\"older continuity of $p \mapsto |p|^m$
instead of its local Lipschitz continuity.

\medskip

The symbol $\Delta_{x_1} u $ stands for the classical Laplacian with
respect to $x_1$ variable and $(-\Delta)_{x_2}^\beta$ for the
fractional Laplacian of order $\beta$ with respect to $x_2$ variable
$$
(-\Delta)_{x_2}^\beta u (x) = \int_{z \in \R^{d_2}} \big( u(x_1, x_2 + z) - u(x) - D_{x_2}u(x)\cdot
z 1_{B_1^{d_2}}(z)\big)\frac{dz}{|z|^{d_2+\beta}}\; .
$$ 
Then, by the results of \cite{bi08} and \cite{bcci12}, this initial
value problem has a unique solution $u(x,t)$ which is Lipschitz
continuous in $x$ for all $t\ge 0$, and $\Z^d$-periodic. By
Theorem~\ref{thm:conv} the solution asymptotically behaves like
$$
u(x,t) = v(x) + \lambda t + o_t(1), \hbox{ as } t\rightarrow \infty, 
$$
where $(\lambda,v)$ is a solution of the stationary ergodic problem
$$
 - \Delta_{x_1} v + (-\Delta)_{x_2}^\beta v + |Dv| = f(x) -\lambda \; \hbox{ in } \R^d.
$$

 Solutions of this initial value problem are H\"older continuous for
  $\beta\in (0,1]$ and Lipschitz continuous for $\beta\in(1,2)$, and
any $m\geq 0$.  However to solve the stationary problem we either
require for the SMP to hold in which case it is necessary that
$\beta>1$ and $m\leq 1$ (see \cite{c12}), or we can use an estimate as
in Lemma\ \ref{lemma:unif_cont} to deal with the superlinear case
$m>1$. Finally, to establish the long time behavior we need to
employ SCP: this holds if $m\leq 1$ for H\"older solutions, whereas
for $m>1$ it holds only if solutions are Lipschitz continuous, hence
$\beta >1$. All in all, the long time behavior is established both in
the sublinear and superlinear case, for fractional exponent
$\beta>1$.
\end{example}

\begin{example}[{\bf Mixed equations with first-order-terms}] \label{ex:mixed_gr}
Consider the initial value problem with mixed local-nonlocal diffusion
and partial gradient terms
\begin{equation*} \left\{
\begin{array}{ll}\medskip
\hspace{-2mm} \partial_t u - a_1(x_1)\Delta_{x_1} u - \mathcal
I_{x_2}[u] + b_1(x_1) |D_{x_1}u|^{k_1} + b_2(x_2) |D_{x_2}u|^{k_2} = f
& ~\hspace{-5mm}\hbox{ in }\R^d\times(0,\infty)
\\ \hspace{-2mm} u(\cdot, 0) = u_0(\cdot) &~\hspace{-5mm} \hbox{ in
}\R^d.
\end{array}\right.
\end{equation*} 
with $u_0, f$ Lipschitz continuous and $\Z^d$-periodic, $a_1(\cdot) =
\sigma^2(\cdot) \ge a_0$ with $\sigma(\cdot)$ a Lipschitz continuous
and $\Z^{d_1}$ periodic function, $b_i\in C^{0,\tau_i}(\R^{d_i})$ for
$\tau_i\in(0,1)$ and $\Z^{d_i}$ periodic, $i=1,2$.  { The exponents
  must satisfy $ k_1\in(0, 2+\tau_1), k_2\in(0,\beta+\tau_2)$.}  The
nonlocal operator $\mathcal I_{x_2}[u]$ is is of L\'evy-It\^o form,
associated to a L\'evy measure $\mu$ satisfying (M1)-(M5), with
fractional exponent $\beta\in(1,2)$
$$ 
\mathcal I_{x_2}[u](x) = \int_{z \in \R^{d_2}} \big( u(x_1, x_2 +
z) - u(x) - D_{x_2}u(x)\cdot z 1_{B_1^{d_2}}(z)\big)\mu(dz)\; .
$$

There exists a solution $u(x,t)$ of this initial value problem  which
is Lipschitz continuous in space.  Once again, by
Theorem~\ref{thm:conv} the solution asymptotically behaves like
\[u(x,t) = v(x) + \lambda t + o_t(1), \hbox{ as } t\rightarrow \infty, \]
where $(\lambda,v)$ is a solution of the stationary ergodic problem
\[ - a_1(x_1)\Delta_{x_1} v - \mathcal I_{x_2}[v]  + b_1(x_1)
|D_{x_1}v|^{k_1} +  b_2(x_2) |D_{x_2}v|^{k_2}  = f(x) -\lambda \;
\hbox{ in } \R^d.\]

We recall that, in \cite{bcci12}, the regularity of solutions is
obtained independently for each set of variables $x_1,x_2$. Roughly
speaking, and taking into account the remarks of
Section~\ref{sec:extensions} for the superlinear cases, this requires
(i) the special form of the equation with no coupling between the
$x_1$ and $x_2$ dependences, and (ii) the $x_1$ and $x_2$ parts of the
equation have both to satisfy  hypotheses which ensure Lipschitz
regularity results and (iii) comparison results. This justifies the
above constraints on $a_1$, $b_i, k_i, \tau_i$ (i=1,2).

We insist on the fact that we can only deal in the variable $x_2$ with
L\'evy-It\^o operators, since uniqueness results were established in
\cite{bi08} for this type of operators only. The general case is still
being left as an important open problem. For this reason, we cannot
have a coefficient of the form $a_2(x_2)$ in front of the nonlocal
diffusion $\mathcal I_{x_2}$ (except for the cases discussed in
Example \ref{ex:composed}), even if, under suitable assumption, we
could still have a regularity result.

Thus, one first sees that the solution is Lipschitz continuous with
respect to $x_1$, for directional gradient terms $b_1(x_1)
|D_{x_1}u|^{k_1}$ having natural growth $k_1\leq 2+\tau_1$, and
H\"older continuous with respect to the $x_2$ variable for directional
gradient terms $b_2(x_2) |D_{x_2}u|^{k_2}$ having natural growth
$k_2\leq \beta+\tau_2$. When $\beta >1$, the solution is globally
Lipschitz continuous.

For the stationary ergodic problem, there are two cases we can treat:
in the {\em sublinear case}, namely $k_1,k_2\leq 1$, we use as before
SMP (we refer the reader to Example \ref{ex:toy} to see that having
H\"older continuous nonlinearities does not create any additional
difficulty). In the {\em superlinear case} we require $k_1,k_2 > 1$
and
\begin{equation}\label{eq:constraint_b}
  b_1(x_1) , b_2(x_2) \geq b_0 > 0\hbox{ for all }x=(x_1,x_2)\in\R^d.
\end{equation}
There is a slight difficulty to prove Lemma \ref{lemma:unif_cont} if
$k_1 \neq k_2$: in order to do it, say in the case when $k_1 < k_2$,
we divide the equation satisfied by $w^\delta$ by $c_\delta^{k_1-1} $
(as we did it with $c_\delta^{m-1}$ for Lemma \ref{lemma:unif_cont}):
according to the above assumption on $b_1,b_2$ the inequality can be
written as
$$ b_0 |D_{x_1}w^\delta |^{k_1} + b_0 c_\delta^{k_2-k_1} |D_{x_2}w^\delta|^{k_2} \leq \cdots\; ,$$
and since we may assume that $c_\delta \geq 1$, we are lead to
$$ b_0 |D_{x_1}w^\delta |^{k_1} + b_0  |D_{x_2}w^\delta|^{k_2} \leq \cdots\; .$$
This allows to get the same conclusion as in Lemma \ref{lemma:unif_cont} since (for proving Claim 1) if $D w^\delta$ is large, then
$$
b_0 |Dw^\delta |^{k_1} \leq 2^{k_1}(b_0 |D_{x_1}w^\delta |^{k_1} +
b_0 |D_{x_2}w^\delta|^{k_2})\; ,
$$ 
because we have either
$|D_{x_1}w^\delta |\geq |Dw^\delta |/2$ or $|D_{x_2}w^\delta |\geq
|Dw^\delta |/2$.

On the other hand, solutions of this equation must satisfy SCP, in
order to have the above asymptotic behavior. To this end we need the
solutions to be Lipschitz continuous (thus $\beta>1$) whenever we deal
with $k_1,k_2\ge 1$.

\end{example}

\begin{example}[{\bf Sublinear vs. Superlinear Gradients Terms}]\label{mixed_subsup}

Consider the initial value problem with mixed local-nonlocal diffusion
and partial gradient terms
\begin{equation*} \left\{
\begin{array}{ll}\medskip
~\hspace{-2mm} \partial_t u - \sigma^2\Delta_{x_1} u + \mathcal I_{x_2}[u] + b_1
|D_{x_1}u|^{k_1} + b_2 |D_{x_2}u|^{k_2} + |Du|^m= f &
\hbox{ in  }\R^d\times
(0,\infty) \\ ~\hspace{-2mm}  u(\cdot, 0) = u_0(\cdot) & \hbox{ in  }\R^d
\end{array}\right.
\end{equation*} 
where $u_0,\sigma, f, b_i$ and $k_i$ are as before. In addition to the
previous example we have an extra full gradient term of growth $m\geq
0$.

The above arguments apply here. We just point out the interplay
between the different gradient terms. When solving the stationary
ergodic problem, we distinguish several different cases.  When $ m
\leq 1$ the equation behaves as in the previous example. Namely either
$k_1,k_2\leq 1$ and the equation is {\em sublinear}, thus we apply SMP
to conclude as usual; or we require $k_1,k_2>1$ in which case we have
the constraint (\ref{eq:constraint_b}).  When $m>1$, we have to
distinguish two situations: either $m>\max(k_1,k_2)$ and hence the
full gradient takes over the other two directional gradients, and in
this case $b_1(\cdot),b_2(\cdot)$ can change sign. Or we can deal with
$m<\min(k_1,k_2)$ and we argue as in Example \ref{ex:mixed_gr}, under
the additional assumption (\ref{eq:constraint_b}).

\end{example}

\bibliographystyle{siam} 
\bibliography{pide}

\end{document}